\documentclass[12pt]{amsart}
\usepackage{fullpage,amssymb}
\usepackage{mathrsfs}
\usepackage{color}
\usepackage{algorithmicx}
\usepackage{algpseudocode}
\usepackage{tikz-cd}

\usepackage{color}   %May be necessary if you want to color links
\usepackage{hyperref}
\hypersetup{
    linktoc=all,     %set to all if you want both sections and subsections linked
    linkcolor=black,  %choose some color if you want links to stand out
}

\newtheorem{theorem}{Theorem}[section]
\newtheorem{corollary}[theorem]{Corollary} 
\newtheorem{lemma}[theorem]{Lemma}
\newtheorem{proposition}[theorem]{Proposition}

\theoremstyle{definition}
\newtheorem{definition}[theorem]{Definition}
\newtheorem{remark}[theorem]{Remark}
\newtheorem{example}[theorem]{Example}

\newcommand{\R}{{\mathbb R}}

\newcommand{\B}{{\mathcal B}}
\newcommand{\C}{{\mathbb C}}

\newcommand{\Q}{{\mathbb Q}}
\newcommand{\Z}{{\mathbb Z}}

\newcommand{\Mat}{\operatorname{Mat}}
\newcommand{\GL}{\operatorname{GL}}
\newcommand{\SL}{\operatorname{SL}}

\newcommand{\VV}{{\mathbb V}}

\newcommand{\nullcone}{\mathcal{N}}
\newcommand{\Sym}{S}

\newcommand{\rk} {\operatorname{rk}}
\newcommand{\g} {\mathfrak{g}}

\newcommand{\tn}{\mathfrak{t}}

\newcommand{\U}{\operatorname{U}}
\newcommand{\SU}{\operatorname{SU}}

\newcommand{\uncramped} {uncramped}
\newcommand{\Supp} {{\rm Supp}}
\newcommand{\wt}{\widetilde}
\newcommand{\isom} {\stackrel{\sim}{\longrightarrow}}
 \newcommand{\sslash}{\mathbin{\mkern-4mu/\mkern-6mu/\mkern-4mu}}

\title{An exponential lower bound for the degrees of invariants of cubic forms and tensor actions}
\author{Harm Derksen and Visu Makam}
\thanks{This material is based upon work supported by the National Science Foundation under Grant No. DMS-1601229 and DMS-1638352}
\begin{document}

\maketitle
\begin{abstract}
Using the Grosshans Principle, we develop a method for proving lower bounds for the maximal degree of a system of generators of an invariant ring.
This method also gives lower bounds for the maximal degree of a set of invariants that define Hilbert's null cone. 
We consider two actions: The first is the action of $\SL(V)$ on $\Sym^3(V)^{\oplus 4}$, the space of $4$-tuples of cubic forms, and the second is the action of $\SL(V) \times \SL(W) \times \SL(Z)$ on the tensor space $(V \otimes W \otimes Z)^{\oplus 9}$.
In both these cases, we prove an exponential lower degree bound for a system of invariants that generate the invariant ring or that define the null cone.   
 %Harm: I modified the abstract a bit
\end{abstract}

\tableofcontents

%For this paper, we will assume the ground field to be $\C$. 
 %Harm: I put this in the intro
\section{Introduction}
For simplicity, we choose the set $\C$ of complex numbers as our ground field, although most results are valid for arbitrary fields of characteristic 0.
Let $V$ be a rational representation of a reductive group $G$ and denote the ring of polynomial functions on $V$ by $\C[V]$.
The group $G$ also acts on $\C[V]$ and the ring of invariants is 
$$\C[V]^G=\{f\in \C[V]\mid \mbox{$g\cdot f=f$ for all $g\in G$}\}.$$
 It is well known that the ring of invariants $\C[V]^G=\bigoplus_{d=0}^\infty \C[V]^G_d$ is a finitely generated graded subring of the polynomial ring $\C[V]$ (see \cite{Haboush,Hilbert1,Hilbert2,Nagata}). All representations in this paper will be rational representations by default. A fundamental question in invariant theory is to describe the generators of an invariant ring and their relations.
 % In several cases, explicit generators of invariant rings have been found over the last centuries.  In other cases, {\em upper} bounds were given for the maximal degree of a system of generators. In this paper our focus will be {\em lower} degree bounds.

Invariant rings play a central role in the Geometric Complexity Theory (GCT) approach to the P vs NP problem. This connection to computational complexity results in new problems in invariant theory, albeit with a different flavor. As one might expect, these problems are more quantitative in nature, asking for how easy or hard the invariant ring is from a computational perspective. There are well understood notions of hardness of computation in computational complexity. We refer to \cite{GCT5} for precise details, as well as numerous conjectures and open problems in invariant theory that are inspired by computational complexity.
%\comm{May be expand on the above?}
From the perspective of GCT, a central problem of interest is the problem of degree bounds for generators.

The problem of finding strong upper bounds for the degrees of generators has been studied. An approach via understanding the null cone was proposed by Popov (see~\cite{Popov1,Popov2}), and improved by the first author, see \cite{Derksen1}. 
The zero set of a set of polynomials $S\subseteq \C[V]$ is
$$
\VV(S)=\{v\in V\mid f(v)=0\mbox{ for all $f\in S$}\}.
$$
Hilbert's null cone $\nullcone\subseteq V$ is defined by $\nullcone=\VV(\bigoplus_{d=1}^\infty \C[V]_d^G)$.
\begin{definition}
We  define $\sigma_G(V)$ to be the smallest integer $D$ such that the non-constant homogeneous invariants of degree $\leq D$ define the null cone,
so 
$$
\textstyle\sigma_G(V)=\min\Big\{D\,\Big|\, \nullcone=\VV\big(\bigoplus_{d=1}^D \C[V]^G_d\big)\Big\}.
$$
General upper bounds for $\sigma_G(V)$ were first given by Popov (see~\cite{Popov1,Popov2}), and improved by the first author in \cite{Derksen1}.

\begin{remark} \label{hsop}
The number $\sigma_G(V)$ can also be defined as the smallest integer $D$ such that $\C[V]^G$ is a finite extension over the subalgebra generated by $\oplus_{d = 0}^D \C[V]^G_d$. 
\end{remark}

We define $\beta_G(V)$ to be the smallest integer $D$ such that invariants of degree $\leq D$ generate $\C[V]^G$, i.e.,
$$
\textstyle\beta_G(V) = \min\Big\{D\,\Big|\, \bigoplus_{d=0}^D \C[V]^G_d \text{ is a generating set for } \C[V]^G\Big\}.
$$
\end{definition}
%\begin{remark}
The number $\beta_G(V)$ can also be seen as the largest degree of a minimal set of (homogeneous) generators for $\C[V]^G$. It is easy to see that
$\beta_G(V)\geq \sigma_G(V)$. The first author showed in \cite{Derksen1} that $\beta_G(V)\leq \max\{2,\frac{3}{8}r\sigma_G(V)^2\}$, where $r$ is the Krull dimension of $\C[V]^G$, which is bounded above by $\dim V$.

%\end{remark}

In this paper, we focus instead on lower bounds. The key idea is to compare two invariant rings via a surjective map between them. 
\begin{lemma} \label{surj.bw.inv}
Suppose $U_1,U_2$ are representations of $G$ and $H$ respectively, such that we have a degree non-increasing surjective homomorphism $\phi: \C[U_1]^G \twoheadrightarrow \C[U_2]^H$. Then we have 
$$
\beta_G(U_1) \geq \beta_H(U_2) \text{ and } \sigma_G(U_1) \geq \sigma_H(U_2).
$$
\end{lemma}

\begin{proof}
It is clear that $\beta_G(U_1) \geq \beta_H(U_2)$ since surjections preserve generating sets. For the null cone, the argument is slightly more involved, but follows from Remark~\ref{hsop} since surjections preserve finite extensions.
\end{proof}

 The source of such surjective maps for us will be Grosshans principle (\cite{Grosshans}). 
\begin{theorem} [Grosshans principle] \label{gross.princ}
Let $W$ be a representation of $G$, and let $H$ be a closed subgroup of $G$. Then we have an isomorphism
$$
\psi:(\C[G]^H \otimes \C [W])^G \longrightarrow \C[W]^H.
$$
\end{theorem}

%First, let us observe that if $H$ is reductive, then $G/H$ is affine by Matsushima's criterion, see \comm{Find reference}. In particular, this means the product $G/H \times W$ is an affine variety. Further, the natural action of $G$ acts on $G/H$ allows us to consider the action of $G$ on the product $G/H \times W$. In the above theorem, the invariant ring $\C[G/H \times W]^G$ denotes the subring of regular functions of $G/H \times W$ that are invariant under the action of $G$.

We will derive the following result from Grosshans principle.

\begin{theorem} \label{main}
Let $V,W$ be representations of $G$. Suppose $v \in V$ is such that its orbit $G \cdot v$ is closed. Let $H = {\rm Stab}_G(v) =\{g \in G \ |\ g \cdot v = v\}$ be a closed reductive subgroup of $G$. Then we have a degree non-increasing surjection
$$
\phi: \C[V \oplus W]^G \twoheadrightarrow \C[W]^H.
$$
In particular, we have
$$
 \beta_G(V \oplus W) \geq \beta_H(W) \text{ and }  \sigma_G(V \oplus W) \geq \sigma_H(W).
$$
\end{theorem}

In order to use this method for finding invariant rings for $G$ with large degree lower bounds, there are mainly three steps, each of which is relatively challenging. First, we have to show that the orbit of a certain point $v$ is closed. Next, we must compute its stabilizer $H$. Finally, we need to find a $G$-representation $W$ for which $\beta_H(W)$ is large. 

We develop the techniques in this paper in a general setup as we believe they are likely useful in many situations. To show that orbits are closed, we will develop a criterion using the moment map (see Theorem~\ref{crit.co}). We will pick our point carefully, so as to ensure that its stabilizer is a torus. For torus actions, it is relatively easier to construct examples with exponential lower bounds. 

One of the main intentions of this paper is to demonstrate that exponential lower bounds can be achieved for fairly simple representations of $\SL_n$. To this end, we are able to prove lower bounds for the action on $4$-tuples of cubic forms. We obtain:

\begin{theorem} \label{lbsln}
Let $V$ be a vector space of dimension $3n$. Then 
$$
\beta_{\SL(V)}(\Sym^3(V)^{\oplus 4}) \geq \sigma_{\SL(V)}(\Sym^3(V)^{\oplus 4}) \geq \textstyle\frac{2}{3} (4^n - 1).
$$
\end{theorem}

We note that $\dim(\Sym^3(V)^{\oplus 4}) = O(n^3)$, and $\dim(\SL(V)) = O(n^2)$. So, the group and the representation are polynomially sized in $n$, while the lower bound for the degree of generators is exponential in $n$.

Another goal of this paper is to gain a better understanding of the computational hardness of the invariant ring for tensor actions. More precisely, consider the action of $\SL(V_1) \times \SL(V_2) \times \dots \times \SL(V_d)$ on $(V_1 \otimes V_2 \otimes \dots \otimes V_d)^{\oplus m}$ defined on each copy of $V_1 \otimes V_2 \otimes \dots \otimes V_d $ by 
$$
(g_1,g_2,\dots,g_d) \cdot v_1 \otimes \dots \otimes v_d = g_1v_1 \otimes \dots \otimes g_dv_d.
$$
A major open problem in complexity is the problem of polynomial identity testing (PIT). A polynomial time algorithm for PIT would be a major step towards the celebrated P vs NP problem, see \cite{KI,GCT5} for details. The null cone membership and orbit closure intersection problems for various invariant rings are closely related to various subclasses of PIT problems, see \cite{GCT5}. 

The vital role of degree bounds is exemplified for the above tensor action in the case $d = 2$ (often called matrix semi-invariants). The `polynomial' degree bounds proved in \cite{DM,DM-arbchar} for matrix semi-invariants were instrumental in giving an algebraic polynomial time algorithm for the null cone membership and orbit closure algorithms in this case, see \cite{DM,DM2,IQS,IQS2}. The algorithm for the null cone problem for matrix semi-invariants gives a polynomial time algorithm for non-commutative rational identity testing. The orbit closure intersection problem solves another subclass of PIT problems in polynomial time. An analytic algorithm over $\Q$ for this subclass appears in \cite{AGOLW}. Despite the analytic nature of the algorithm, the polynomiality of degree bounds are crucial to show that the algorithm runs in polynomial time! In summary, degree bounds are an essential component in understanding the challenges and boundaries of algorithmic efficiency.

The cases when $d \geq 3$ have also been the subject of recent interest, see for example \cite{BGOWW}. We prove exponential lower bounds for these tensor actions. We show:

\begin{theorem} \label{tensor-lbs}
Suppose $V,W,Z$ are vector spaces of dimension $3n$. Then, for the tensor action of $G = \SL(V) \times \SL(W) \times \SL(Z)$ on $(V \otimes W \otimes Z)^{\oplus 9}$, we have
$$
\beta_G(V) \geq \sigma_G(V) \geq 4^n-1
$$
\end{theorem}

Again, let us point out that the dimension of the group and representation are polynomial in $n$, but the lower bounds on the degree of generation is exponential in $n$.

\section{Preliminaries from linear algebra} \label{prelim}
We will first setup some preliminaries from linear algebra. An $n \times m$ matrix $A$ should be interpreted as a linear map $A: \Q^m \rightarrow \Q^n$. The null space of $A$ is defined as 
$$
\mathcal{Z}(A) = \{v \in \Q^m\ |\ Av = 0\}.
$$

We will be interested in non-negative integral points in the null space. So, we define 
$$
\mathcal{I}(A) = \mathcal{Z}(A) \cap \Z_{\geq 0}^m.
$$

Observe that $\mathcal{I}(A)$ is a monoid under addition. Further, we will be interested in the minimal generators of the monoid $\mathcal{I}(A)$. So, we define
$$
\mathcal{GI}(A) = \{v \in \mathcal{I}(A)\ |\ v \neq w_1 + w_2 \ \forall\ w_1,w_2 \in \mathcal{I}(A)\setminus\{0\}\}.
$$

It is easy to see that $\mathcal{GI}(A)$ is a minimal generating set for the monoid $\mathcal{I}(A)$.

We will be interested in computing this in two specific cases. The first is the $n \times (n+1)$ matrix
\begin{equation}
M = \begin{pmatrix}  
1 & 0 &\dots & \dots &  0 & -4 & 3 \\
-4 & 1 & \ddots &  & 0 & 0  & 0\\
0 & -4 &  \ddots & \ddots & \vdots & \vdots & \vdots\\
\vdots & 0 & \ddots & \ddots & \vdots & \vdots & \vdots \\
\vdots & \vdots & \ddots & \ddots & 1 & 0 & 0\\
0 & \dots & 0 & 0 & -4 & 1 & 0 
\end{pmatrix}
\end{equation}

\begin{lemma}
We have $\mathcal{Z}(M) = \Q \cdot \left(1,4,16,\dots,4^{n-1}, \frac{4^n -1}{3}\right)^t$.
\end{lemma}

\begin{proof}
It is clear that the matrix $M$ has full rank, i.e., $\rk(M) = n$. By the rank-nullity theorem, we know that $\mathcal{Z}(M)$ is $1$-dimensional. The lemma follows by checking that $M$ kills $\left(1,4,16,\dots,4^{n-1}, \frac{4^n -1}{3}\right)^t$.
\end{proof}

\begin{corollary} \label{gi.calc.m}
The set $\mathcal{GI}(M)$ consists of only one vector. Further, we have
$$
\mathcal{GI}(M) = \bigg\{ \left(1,4,16,\dots,4^{n-1}, \frac{4^n -1}{3}\right)^t \bigg\}$$

\end{corollary}

\begin{proof}
Since $\mathcal{Z}(M)$ is $1$-dimensional, the set $\mathcal{GI}(M)$ consists of at most one element. This will be smallest non-negative integral element in $\mathcal{Z}(M)$, and this is the one given in the statement of the corollary.
\end{proof}

The second case we will be interested in is the $3n \times (3n-1)$ matrix 
$$
N = \begin{pmatrix}
B & I_3 &  &&& \\
 & P & I_3  &&& \\
 && P & \ddots  &&  \\
 &&& \ddots & I_3 & \\
 &&&& P & A\\
\end{pmatrix},
$$

Where $$
A = \begin{pmatrix} 1 \\ 1 \\1 \end{pmatrix}, P = \begin{pmatrix} -2 & -1 & -1 \\ -1 & -2 & -1 \\ -1 & -1 & -2 \end{pmatrix}, I_3 = \begin{pmatrix} 1 & & \\ & 1 & \\ &&1 \end{pmatrix}, \text{ and } B = \begin{pmatrix} -2 \\ -2 \\ -2 \end{pmatrix}.
$$

\begin{lemma}
We have $\mathcal{Z} (N) =  \Q \cdot \left(1,2,2,2,8,8,8,\dots, 2^{2n-3},2^{2n-3},2^{2n-3}, 2^{2n-1}\right)^t$
\end{lemma}

\begin{proof}
Suppose $v = (v_1,\dots,v_{3n-1})$ is such that $N v = 0$. Let us look at this as a system of $3n$ equations in $3n-1$ variables. As is well understood, each row gives one equation. Let us assume $v_1 = \alpha$. Now, we will go through the equations corresponding to the rows from top to bottom to deduce what $v_i$ have to be for $i \geq 1$. 

The first three rows imply that $v_2 = v_3 = v_4 = 2 \alpha$. The fourth row implies that $v_5 = 2v_2 + v_3 + v_4 = 4(2 \alpha)$. Similarly the fifth and sixth rows imply $v_6 = v_7 =  8 \alpha$. The process repeats until we get $v_{3n-4} = v_{3n-3} = v_{3n-2}  = 2^{2n-3} \alpha$. The last three equations all imply that $v_{3n-1} = 2^{2n-1} \alpha$.  In other words, we have $v = \alpha \cdot \left(1,2,2,2,8,8,8,\dots, 2^{2n-3},2^{2n-3},2^{2n-3}, 2^{2n-1}\right)^t$
\end{proof}

Using a similar argument to the case of $M$, we get:

\begin{corollary}
The set $\mathcal{GI}(N)$ consists of only one vector. Further, we have
$$
\mathcal{GI}(N) = \bigg\{\left(1,2,2,2,8,8,8,\dots, 2^{2n-3},2^{2n-3},2^{2n-3}, 2^{2n-1}\right)^t\bigg\}$$

\end{corollary}

\section{Invariants for torus actions}
We will briefly recall invariant theory for torus actions. Let $T = (\C^*)^n$ be an $n$-dimensional (complex) torus.  A group homomorphism $T \rightarrow \C^*$ is called a character of $T$. Given two characters $\lambda,\mu:T \rightarrow \C^*$, we define a character $\lambda + \mu :T \rightarrow \C^*$ defined by $(\lambda + \mu) (t) = \lambda(t) \mu(t)$. With this operation, the set of characters of $T$ form a group called the character group, which we denote by $\mathcal{X}(T)$. 

To each $\lambda = (\lambda_1,\dots,\lambda_n) \in \Z^n$, we can associate a character also denoted $\lambda$ by abuse of notation. The character $\lambda: T \rightarrow \C^*$ is defined by $\lambda(t) = \prod_{i=1}^n t_i^{\lambda_i}$. This gives an isomorphism of groups
$\Z^n \xrightarrow{\sim} \mathcal{X}(T)$. Characters of the torus are often called weights, and we will use this terminology as well.

Let $V$ be a rational representation of $T$. We make the identification $\mathcal{X}(T) = \Z^n$. For a weight $\lambda \in \Z^n$, the weight space $V_{\lambda} = \{v \in V\ |\ t \cdot v = \lambda(t) v \ \forall t \in T\}$. A vector $v \in V_{\lambda}$ is called a weight vector of weight $\lambda$. Any representation $V$ is a direct sum of its weight spaces, i.e., $V = \oplus_{\lambda \in \Z^n} V_{\lambda}$. In other words, we have a basis consisting of weight vectors. 

Let $\mathcal{E} = (e_1,\dots,e_m)$ be an ordered basis of $V$ consisting of weight vectors. Further, suppose each $e_i$ is a weight vector of weight $\lambda_i$. Let $x_1,\dots,x_m$ denote the coordinate functions with respect to  the basis $e_1,\dots,e_m$. The following are well known: 

\begin{enumerate}
\item A monomial $x^v = x_1^{v_1} x_2^{v_2} \dots x_m^{v_m}$ is an invariant monomial if and only if $ \sum_i v_i \lambda_i = 0$.
\item The ring of invariants $\C[V]^T$ is linearly spanned by such invariant monomials.
\end{enumerate}

We will rewrite the above results in a slightly different language. We will first need a definition.

\begin{definition} \label{wt.matrix}
Let $V$ be a representation of $T$ with an (ordered) weight basis $\mathcal{E} = (e_1,\dots,e_m)$. Further, suppose each $e_i$ is a weight vector of weight $\lambda_i$. Define $M_{\mathcal{E}}(V)$ to be the $n \times m$ matrix whose $i^{th}$ column is $\lambda_i$, i.e.,
$$
M_{\mathcal{E}}(V) := \begin{pmatrix} 
| & | & \dots & | \\
\lambda_1 & \lambda_2 & \dots &  \lambda_m \\
| & | & \dots & | 
\end{pmatrix}
$$
\end{definition}

\begin{remark}
For a different choice of ordered weight basis $\mathcal{E}'$, the matrix $M_{\mathcal{E}'}(V)$ is obtained by a permutation of the columns of $M_{\mathcal{E}}(V)$. This is because the formal sum of the columns (i.e., $\sum_i e^{\lambda_i}$) is called the character of the representation $V$ and independent of the choice of weight basis.
\end{remark}

\begin{proposition}
Let $V$ be a representation of $T$. Let $\mathcal{E} = (e_1,\dots,e_m)$ be a weight basis, and let $x_1,\dots,x_m$ be the corresponding coordinate functions. Then
\begin{enumerate}
\item For $v = (v_1,\dots,v_m) \in \mathcal{I}(M_\mathcal{E}(V))$, $x^v = x_1^{v_1} \dots x_m^{v_m}$ is an invariant monomial;
\item The set $\{x^v\ |\ v \in \mathcal{I}(M_\mathcal{E}(V)) \}$ is a $\C$-linear spanning set of invariants;
\item The set $\{x^v \ |\ v \in \mathcal{GI}(M_{\mathcal{E}}(V)) \}$ is a minimal set of generators for $\C[V]^T$.
\end{enumerate}
\end{proposition}

\begin{proof}
The first two statements is simply a rephrasing of the discussion before Definition~\ref{wt.matrix}. The last follows from the fact that for any matrix $A$, the set $\mathcal{GI}(A)$ is a minimal generating set for the monoid $\mathcal{I}(A)$.
\end{proof}

\begin{proposition} \label{tor.inv.M}
Let $T$ act on $V = \C^{n+1}$ such that for some weight basis $\mathcal{E}$, we have $M_{\mathcal{E}}(V) = M$, the matrix in Section~\ref{prelim}. Then, we have
$$
\beta_T(V) \geq \sigma_T(V) \geq \textstyle\frac{2}{3} (4^n - 1).
$$
\end{proposition}

\begin{proof}
Let $\mathcal{E} = e_1,\dots,e_{n+1}$. Let $x_1,\dots,x_{n+1}$ be the coordinates with respect to this basis. From the above proposition, we know that $\{x^v\ |\ v \in \mathcal{GI}(M)\}$ is a minimal set of generators for the invariant ring. Corollary~\ref{gi.calc.m} tells us that $\mathcal{GI}(M)$ consists of precisely one element. The corresponding monomial is $f := x_1x_2^4x_3^{16} \dots x_n^{4^{n-1}} x_{n+1}^{(4^n-1)/3}$. To summarize, we have $\C[V]^T = \C[f]$. 

It is clear that $f$ has degree $(1 + 4 + \dots 4^{n-1} + \frac{4^n-1}{3}) = \frac{2}{3} (4^n - 1)$. This already gives us that $\beta_T(V) \geq \frac{2}{3} (4^n-1)$. Now, consider $v = e_1 + \dots + e_{n+1}$. Then, we have $f(v) \neq 0$, so $v$ is not in the null cone. Since there are no non-constant homogeneous invariants of smaller degree, it follows that $\sigma_T(V) \geq \frac{2}{3} (4^n-1)$.

\end{proof}

A similar argument gives the following:

\begin{proposition} \label{tor.inv.N}
Let $T = (\C^*)^{3n}$ act on an $V = \C^{3n-1}$ such that for some weight basis $\mathcal{E}$, we have $M_{\mathcal{E}}(V) = N$, the matrix in Section~\ref{prelim}. Then, we have
$$
\beta_T(V) \geq \sigma_T(V) \geq 4^n-1
$$
\end{proposition}

\begin{proposition} \label{tor.inv.compare}
Suppose $V \subseteq W$ are two representations of $T$, then 
$$
\beta_T(V) \leq \beta_T(W) \text{ and } \sigma_T(V) \leq \sigma_T(W).
$$
\end{proposition}

\begin{proof}
Representations of tori are completely reducible, so we have $W = V \oplus V'$, where $V'$ is also a subrepresentation of $T$. The inclusion $V \hookrightarrow W$ gives a surjection $\pi: \C[W] \rightarrow \C[V]$ that is clearly degree non-increasing. It is easy to check that $\pi$ descends to a map of invariant rings $\C[W]^T \rightarrow \C[V]^T$. We claim that this is a surjection. Indeed, for $f \in \C[V]^T$, define $\wt{f}$ by $\wt{f}(v,v') = f(v)$ for all $(v,v') \in V \oplus V' = W$. Clearly $\wt{f} \in \C[W]^T$ and $\pi(\wt{f}) = f$. The fact that the surjection $\pi: \C[W]^T \twoheadrightarrow \C[V]^T$ is degree non-increasing implies both statements by Lemma~\ref{surj.bw.inv}.
\end{proof}

\section{Grosshans principle}
First let us note that for any vector space $U$, the coordinate ring $\C[U] = \Sym(U^*)$ is a polynomial ring, and hence we have a grading $\C[U] = \oplus_{d=0}^{\infty} \C[U]_d$. We will call this the polynomial grading.

For any vector space $W$, and any ring $R$, we can define a grading on $R \otimes \C[W]$ by setting $(R \otimes \C[W])_d = R \otimes \C[W]_d$. We will call this the $W$-grading.

We will first give an outline of the proof of Grosshans principle, i.e., Theorem~\ref{gross.princ}. 
\begin{proof} [Proof of Theorem~\ref{gross.princ}]

Consider the action of $G \times H$ on $G \times W$ by 
$$
(g',h') \cdot (g,w) = (g'g(h')^{-1}, h' \cdot w).
$$
Let us compute the ring of invariants $\C[G \times W]^{G \times H}$. First, let us observe that the action of $G$ is trivial on $W$, so we have $\C[G \times W]^G = (\C[G] \otimes \C[W])^G = \C[G]^G \otimes \C[W] = \C[W]$. Hence, we have
$$
\C[G \times W]^{G \times H} = (\C[G \times W]^G)^H = \C[W]^H.
$$

Now, let us consider another action of $G \times H$ on $G \times W$ given by 
$$
(g,w) \cdot (g,w) = (g'g(h')^{-1}, g' \cdot w).
$$
 In this case, $H$ acts trivially on $W$, so we have $\C[G \times W]^H = (\C[G] \otimes \C[W])^H = \C[G]^H \otimes \C[W]$. Hence, we have 
$$
\C[G \times W]^{G \times H}  = (\C[G \times W]^H)^G = (\C[G]^H \otimes \C[W])^G.
$$

We have both sides of Grosshans principle, and now we need to relate them. Consider the map
\begin{align*}
\psi : G \times W  \rightarrow  G \times W\\
(g,w) \mapsto  (g,g\cdot w)
\end{align*}

This gives a map on the coordinate rings, which we will also denote by  
$\psi: \C[G \times W] \isom \C[G \times W]$. There is a $W$-grading on $\C[G \times W]$ because $\C[G \times W] = \C[G] \otimes \C[W]$. Since $G$ acts by linear transformations, the map $\psi$ preserves the $W$-grading.

Now, observe that $\psi$ takes the first action of $G \times H$ to the second action of $G \times H$. In particular, this means that the map $\psi$ restricts to an isomorphism of the invariant rings
$$
\psi: (\C[G]^H \otimes \C[W])^G \rightarrow \C[W]^H.
$$

Observe that $\C[W]^H$ is a $W$-graded subring of $\C[G \times W]$. Since $\psi$ is an isomorphism that preserves the $W$-grading, $(\C[G]^H \otimes \C[W])^G$ is also a $W$-graded subring.
\end{proof}

Let us recall Matsushima's criterion, see \cite{Matsushima, Bial}.

\begin{theorem} [Matsushima] \label{Matsu}
Let $G$ be a reductive group, and let $H$ be a closed subgroup. Then $H$ is reductive if and only if $G/H$ is an affine variety.
\end{theorem}

An immediate consequence is the following:

\begin{corollary}
Let $G$ be a reductive group, and let $H$ be a closed reductive subgroup. Then we have an isomorphism
$$
\C[G]^H \isom \C[G/H].
$$
\end{corollary}

We need one more lemma before we prove Theorem~\ref{main}.
\begin{lemma} \label{need.char0}
Let $Y$ be a variety with an action of a reductive group $G$. Suppose $X$ is a closed $G$-stable subvariety of $X$, then we have a surjection $\C[Y]^G \twoheadrightarrow \C[X]^G$.
\end{lemma}

\begin{proof}
Since $G$ is reductive, we have Reynolds operators $R_Y: \C[Y] \twoheadrightarrow \C[Y]^G$ and $R_X: \C[X] \twoheadrightarrow \C[X]^G$. Let $i: X \hookrightarrow Y$ denote the inclusion map, and let the pull back map on the coordinate rings be $i^*: \C[Y] \twoheadrightarrow \C[X]$. In the following diagram, the horizontal arrows are Reynolds operators and the vertical arrows are given by $i^*$.

\begin{center}
\begin{tikzcd}
\C[Y] \arrow[r] \arrow[d]
& \C[Y]^{G} \arrow[d] \\
\C[X] \arrow[r]
& \C[X]^G
\end{tikzcd}
\end{center}
The above diagram commutes. To see this, let us observe that $\C[Y]$ can be decomposed as a direct sum of irreducibles. So, it suffices to see that the diagram commutes for each isotypic component. The isotypic components for non-trivial representations get killed by the Reynolds operators, so both directions send them to zero. The Reynolds operators act by identity on trivial representations. So, in either direction the isotypic component for the trivial representation is only subject to $i^*$.

Hence, the map $i^*: \C[Y]^G \rightarrow \C[X]^G$ must be a surjection since the other three maps are surjections.
\end{proof}

\begin{proof} [Proof of Theorem~\ref{main}]
By the above discussion, the Grosshans principle in this case reads as:
$$
\C[G/H \times W]^G \isom (\C[G]^H \otimes \C[W])^G \isom \C[W]^H.
$$

Observe that $G/H \cong G \cdot v$ as affine varieties \footnote{This follows essentially from Zariski's main theorem, see for eg \cite[Theorem~25.1.2(iv)]{TY}}, and thus we have
$$
G/H \times W \isom G \cdot v \times W \hookrightarrow V \oplus W.
$$

This gives a surjection of invariant rings $\C[V \oplus W]^G \twoheadrightarrow \C[G/H \times W]^G$ by the above lemma. Combining with the above discussion, we have:
$$
\phi: \C[V \oplus W]^G \twoheadrightarrow \C[G/H \times W]^G \isom \C[W]^H
$$

Recall the $W$-grading on $\C[G/H \times W]$. We also have a $W$-grading on $\C[V \oplus W]$. The surjection $\C[V \oplus W] \twoheadrightarrow \C[G/H \times W]$ is degree non-increasing in the $W$-grading, and hence so is $\phi$.

The polynomial grading and $W$-grading are different on $\C[V \oplus W]$. If $f \in \C[V \oplus W]$ is homogeneous in degree $d$ in the polynomial grading, then $f$ need not be homogeneous in the $W$-grading. However, the homogeneous components of $f$ in the $W$-grading will all be in degrees $\leq d$. On the other hand, the $W$-grading and the polynomial grading on $\C[W]^H$ agree. In particular this means that the surjection $\phi:\C[V \oplus W]^G \twoheadrightarrow  \C[W]^H$ is degree non-increasing even when we consider the polynomial grading on $\C[V \oplus W]^G$. Thus we can apply Lemma~\ref{surj.bw.inv} to deduce:
$$
\beta_G(V \oplus W) \geq \beta_H(W) \text{ and } \sigma_G(V \oplus W) \geq \sigma_H(W).
$$
\end{proof}

\section{Root systems and Invariant forms} \label{rs and inv.forms}
Let $G$ be a complex reductive group, and $K$ a maximal compact subgroup, also called a compact real form. Let $T_{\R}$ be a (real) maximal torus of $K$. Being a real torus means that $T_\R \cong S_1^n$ for some $n \in \Z_{\geq 0}$, where $S_1 = \{z \in \C\ |\ |z| = 1\}$. Let $T$ denote the complexification of $T_\R$. Then $T$ is a (complex) maximal torus for $G$. We denote the Lie algebra of $T$ by $\tn$. 

For any representation $V$ of $G$, we can view it as a representation of $T$, and hence we get a weight space decomposition 

$$
V = \bigoplus_{\lambda \in \mathcal{X}(T)} V_{\lambda}.
$$

There is a natural way to view $\mathcal{X}(T)$ as a subset of $\tn^*$. Indeed, let $T = (\C^*)^n$, and consequently its Lie algebra $\tn = \C^n$ where the lie bracket is identically zero. Let $v_1,\dots,v_n$ be the standard basis for $\C^n$, and consider the dual basis $e_1,\dots,e_n \in (\C^n)^* = \tn^*$.

Then the correspondence $\lambda = (\lambda_1,\dots,\lambda_n) \longleftrightarrow \sum_{i=1}^n \lambda_i e_i$ allows us to view $\mathcal{X}(T)$ as a subset of $\tn^*$. This is indeed natural because for the character $\lambda: T = (\C^*)^n \rightarrow \C^*$ given by $(t_1,\dots,t_n) \mapsto \prod_i t_i^{\lambda_i}$, we get a map on the Lie algebras $\lambda: \tn = \C^n \rightarrow \C$ given by  $(x_1,\dots,x_n) \mapsto \sum_i \lambda_ix_i$.

An action of $T$ on $V$ gives an action of the Lie algebra $\tn$ on $V$. For $\lambda \in \tn^*$, we can define $V_\lambda = \{v \in V\ | \ X \cdot v = \lambda(X) v \ \forall X \in \tn\}$. Unless $\lambda \in \mathcal{X}(T)$, we must have $V_\lambda = 0$. Further, both definitions of $V_\lambda$ agree.

Hence, one might also write the above decomposition as 
$$
V = \bigoplus_{\lambda \in \tn^*} V_{\lambda}.
$$

\begin{remark}
As discussed above, there are two ways to view the weight space decomposition. Both are well-known and standard, and we will freely switch between the two as needed.
\end{remark}

We make a convenient definition.
\begin{definition}
For any $v \in V$, the decomposition $v = \sum_{\lambda} v_{\lambda}$ with $v_{\lambda} \in V_{\lambda}$ is called the weight decomposition of $v$. The weight decomposition is unique. Further the set $\{\lambda\ |\ v_{\lambda} \neq 0\}$ is called the support of $v$, and we denote it by $\Supp(v)$.
\end{definition}

The following two examples are meant for readers unfamiliar with root systems, and can be skipped by experts. In Remark~\ref{B-defns}, we develop notation that will be helpful at various stages of paper.

\begin{example} \label{GLnstart}
Suppose $G = \GL_n(\C)$. Consider $\U_n(C) = \{X \in \GL_n(\C)\ |\ XX^\dag = I \}$ the unitary group of matrices, where $X^\dag$ denotes the conjugate transpose of $X$. Then $K = \U_n(C)$ is a compact real form for $\GL_n(\C)$. Let ${\rm diag} (a_1,\dots,a_n)$ denote the diagonal $n \times n$ matrix with diagonal entries $a_1,\dots,a_n$. The real torus $T_\R = \{{\rm diag}(a_1,\dots,a_n)\ |\ a_i \in \C, |a_i| = 1 \}$ is a (real) maximal torus of $K$, and its complexification $T = \{ {\rm diag} (a_1,\dots,a_n)\ |\  a_i \in \C \}$ is a (complex) maximal torus, and the Lie algebra of $T$ is $\tn =  \{{\rm diag} (x_1,\dots,x_n)\ |\ x_i \in \C\}$. Let $\wt{e}_i \in \mathcal{X}(T)$ be defined by $\wt{e}_i \cdot {\rm diag}(t_1,\dots,t_n) = t_i$. Then, for the action of $\GL_n$ on $\C^n$ by left multiplication, the standard basis vector $e_i$ is a weight vector with weight $\wt{e}_i$. The weights $\wt{e}_i$ form a basis of $\tn^*$.
\end{example}

\begin{example} \label{SLnstart}
Suppose $G = \SL_n(\C)$. Then $K = \SU_n(\C) = \{X \in U_n(\C)\ |\ \det(X) = 1 \}$ is a compact real form, $T_\R = \{ {\rm diag}(a_1,\dots,a_n)\ |\ a_i \in \C, |a_i| = 1, \prod_i a_i = 1\}$ is a (real) maximal torus, its complexification $T =  \{ {\rm diag}(a_1,\dots,a_n)\ |\ a_i \in \C,  \prod_i a_i = 1\}$ is a (complex) maximal torus, and its Lie algebra is $\tn = \{{\rm diag} (x_1,\dots,x_n)\ |\ x_i \in \C, \sum_i x_i = 0\}$. The condition $\sum_i x_i = 0$ is really just asking for the trace of the matrix to be $0$. Let $\wt{e}_i \in \mathcal{X}(T)$ be defined again by $\wt{e}_i \cdot {\rm diag} (t_1,\dots,t_n) = t_i$. Then, for the action of $\SL_n$ on $\C^n$ by left multiplication, the standard basis vector $e_i$ is a weight vector with weight $\wt{e}_i$. The weights $\wt{e}_i$ do not form a basis. They satisfy one relation, i.e, $\sum_i \wt{e}_i = 0 \in \tn^*$.
\end{example}

Let us reformulate the above examples. 
\begin{remark} \label{B-defns}
Suppose $G = \GL(V)$ or $\SL(V)$, with $\B$ a basis for $V$. Then, using the basis, we can identify $\GL(V)$ (resp. $\SL(V)$) with $\GL_n$ (resp. $\SL_n$). With this identification, we can define $K_\B,T_{\R,\B}, T_\B, \tn_\B$ as in the above examples. Under these choices, $\B$ consists of weight vectors. Let us denote the weight of $b \in \B$ by $\wt{b}$. The set $\{ \wt{b}\ |\ b \in \B\}$ forms a basis for $\tn_\B^*$ if $G = \GL(V)$, and satisfy one relation. i.e., $\sum_{b \in \B} \wt{b} = 0$ if $G = \SL(V)$.
When the group is not clear, we will write $K_{G,\B}, T_{G,\B}$ etc.
\end{remark}

\subsection{Invariant forms} \label{Section.inv.forms}
For this section, let $G$ be a complex reductive group, $K$ a compact real form, $T_\R$ a (real) maximal torus in $K$ and $T$ the complexification of $T_\R$.

\begin{proposition} \label{com.inv.exist}
Let $V$ be any representation of $G$. Then there is a positive definite $K$-invariant hermitian form on $V$ for which the weight spaces are pairwise orthogonal.
\end{proposition}

\begin{proof}
Let $\rho:G \rightarrow \GL(V)$ define the representation. There is a positive definite hermitian form $\left<-,-\right>$ on $V$ such that $\rho(K) \subseteq U(V)$, where $U(V)$ denotes the unitary group with respect to $\left<-,-\right>$. This is the statement of Weyl's unitary trick. Now, $\rho(T_\R) \subseteq \rho(K) \subseteq U(V)$ is a subtorus of $U(V)$, and hence $\rho(T_\R)$ is contained in a (real) maximal torus $H_\R$ of $U(V)$.
All maximal tori of $U(V)$ are conjugate to each other. In other words, there is an orthonormal basis $\B$ of $V$ such that $K_{\GL(V),\B} = U(V)$ and $T_{\GL(V),\R,\B} = H_\R$.

The basis vectors $b \in \B$ are weight vectors for $H_\R$ and hence for $T_\R$. Since the basis vectors are orthonormal, the weight spaces must be orthogonal.
\end{proof}

\begin{definition} [$(K,T)$-compatible form]
For a representation $V$ of $G$, we call a positive definite $K$-invariant hermitian form $(K,T)$-compatible if the weight spaces are orthogonal.
\end{definition}

The above proposition can now be reformulated as:

\begin{corollary}
Let $V$ be a representation of $G$. Then there exists a $(K,T)$-compatible form on $V$.
\end{corollary}

\begin{definition}
Let $V$ be a vector space with basis $\B$. Let $W$ be a representation of $GL(V)$. Then a $\B$-compatible form on $W$ is defined as an $(K_{\GL(V),\B},T_{\GL(V),\B})$-invariant form. 
\end{definition}

\begin{remark}
If $W$ is a representation of $GL(V)$, then it is also a representation of $SL(V)$. A $\B$-compatible form on $W$ is also a $(K_{\SL(V),\B}, T_{\SL(V),\B})$-form. The converse is not always true.
\end{remark}

\begin{definition} [direct sum form]
Suppose $W_i$ is a vector space with a bilinear form $\left<-.-\right>_i$ for $i = 1,2$. Then we define the direct sum form $\left<-,-\right>$ on $W_1 \oplus W_2$ by 
$$
\left<(a,b),(c,d)\right> = \left<a,c\right>_1 \left<b,d \right>_2.
$$
\end{definition}

The following lemma is straightforward.

\begin{lemma} \label{sum.form}
Suppose $W_1$ and $W_2$ are representations of $G$ with positive definite $K$-invariant hermitian forms. Then the direct sum form gives a positive definite $K$-invariant hermitian form on $W_1 \oplus W_2$. Further, under this form, $W_1$ is orthogonal to $W_2$.
\end{lemma}

\subsection{Root systems}
For a complex reductive group $G$, let $K$ be a compact real form, $T_\R$ a maximal torus in $K$ and $T$ its compexification. Let $\mathfrak{g}$ denote the Lie algebra of $G$. There is an exponential map $\exp: \g \rightarrow G$. There is a natural action of the group $G$ acts on $\g$ called adjoint action. The adjoint action is given by $g \cdot X = \frac{d}{dt} g \cdot \exp(tX) \cdot g^{-1}$. 
Since $T$ is a subgroup of $G$, we get an action of $T$ on $\g$. This gives a decomposition of $\g$ into weight spaces with respect to $T$, i.e.,  
$$
\g = \tn \oplus\bigoplus_{\alpha\in\Phi} \g_\alpha.
$$
Let us explain the terms. For each weight $\beta \in \tn^*$, the weight space $\g_\beta$ consists of all the weight vectors in $\g$ of weight $\beta$. More precisely, $\g_\beta = \{X \in \g\ |\ t \cdot X = \beta(t) X\  \forall t \in T\}$. Since $\g$ is finite dimensional, only a finitely many of these weight spaces are non-zero. The weight space corresponding to the $0$ weight is just $\tn$, i.e., $\g_0 = \tn$. The set of non-zero weights $\beta$ for which the weight space $g_\beta$ is non-zero form a finite collection of vectors in $\mathcal{X}(T) \subset \tn^*$ called the root system, which we denote by $\Phi$. This explains the above decomposition. 

The above decomposition of $\g$ is often called the root space decomposition. Root systems have a very rich structure, and have been explored extensively from algebraic, geometric and combinatorial points of view. We refer to \cite{Humphreys} for an algebraic introduction. 

\begin{example} \label{glnrs}
Let $G = \GL_n$. We continue with the choices for $K,T,\tn$ etc from Example~\ref{GLnstart}. The Lie algebra of $G = \GL_n$ is $\g = \Mat_{n,n}$. Let $E_{i,j}$ denote the elementary $n \times n$ matrix with a $1$ in the $(i,j)^{th}$ entry and $0$'s everywhere else. The set $\{E_{i,j}\}_{1 \leq i,j \leq n}$ is a weight basis for $\Mat_{n,n}$. The torus element $t = {\rm diag}(t_1,\dots,t_n)$ acts on $E_{i,j}$ by the formula
$$
t \cdot E_{i,j} = t_it_j^{-1} E_{i,j}
$$
Recall that $e_1,\dots,e_n$ is a basis for $\tn^*$, where the formula $e_i (t) = t_i$ defines the weight $e_i \in \tn^*$. Hence, the weight $e_i-e_j$ is given by the formula $(e_i - e_j)(t) = t_it_j^{-1}$. In particular, $E_{i,j}$ is a weight vector for the weight $e_i-e_j$. Hence, the obvious decomposition 
$$
\g = \Mat_{n,n} = \tn \oplus \bigoplus_{i,j} \C E_{i,j}
$$
is indeed the weight space decomposition. This means that the root system $\Phi$ consists of weights of the form $e_i - e_j$ for $i \neq j$, i.e., $\Phi = \{e_i - e_j \ |\ 1\leq i,j \leq n,  i \neq j\}$. 
\end{example}

\begin{example} \label{slnrs}
Let $G = \SL_n$. We continue with the choices for $K,T,\tn$ etc from Example~\ref{SLnstart}. The Lie algebra of $G = \SL_n$ is $\g = \{X \in \Mat_{n,n}\ |\ {\rm Tr}(X) = 0 \}$. Recall that in this case $e_1,\dots,e_n$ satisfy one linear relation, i.e., $\sum_i e_i = 0$. Again, the root system $\Phi = \{e_i - e_j \ |\ 1\leq i,j \leq n,  i \neq j\}$. 

\end{example}

\begin{remark} \label{b-rootsys}
If we started with $G = \GL(V)$ or $\SL(V)$, and a basis $\B$ of $V$ and made all the standard choices as in Remark~\ref{B-defns}, the root system would be 
$\Phi = \{\wt{b} - \wt{b'}\ |\ b,b' \in \B, b \neq b'\}$
\end{remark}

We make some useful definitions to aid in formulating later statements.

\begin{definition} [Root adjacent]
We say two weights $\lambda,\mu \in \tn^*$ are root adjacent if $\lambda - \mu \in \Phi$.
\end{definition}

\begin{definition} [\uncramped\  sets of weights] \label{uncramped}
A subset of weights $I \subseteq \tn^*$ is called \uncramped\  if no pair of weights in $I$ is root adjacent.
\end{definition}

\subsection{Products of root systems}
Suppose $G_1,\dots,G_d$ are connected reductive groups. Suppose for each $i$ that $K_i,T_i$ are choices of maximal compact subgroups and maximal tori. Let $\Phi_i$ be the root system for $G_i$ corresponding to these choices. Then $K:= K_1 \times K_2 \times \dots \times K_d$ (resp. $T:= T_1 \times \dots \times T_d$) is a maximal compact subgroup (resp. maximal torus) for $G:= G_1 \times G_2 \times \dots \times G_d$. Let $\tn = \tn_1 \times \tn_2 \dots \times \tn_d$ is the Lie algebra of $T$, where $\tn_i$ denotes the Lie algebra of $T_i$. Observe that $\Phi_i \subseteq \tn_i^* \subset \tn^*$. The following is straightforward.

\begin{lemma}
The set $\Phi_1 \cup \Phi_2 \cup \dots \cup \Phi_d$ is the root system for $G$. 
\end{lemma}

Let $\underline{\lambda} = (\lambda_1,\lambda_2,\dots,\lambda_d)$ and $\underline{\mu} = (\mu_1,\dots,\mu_d)$ be two weights in $\tn^*$. 

\begin{corollary} 
Suppose $\lambda_t \neq \mu_t$ for at least two choices of $t \in \{1,2,\dots,d\}$. Then $\underline{\lambda}$ and $\underline{\mu}$ are not root adjacent.  
\end{corollary}

Suppose $V_i$ is a representation of $G_i$ for each $i$. Then $V = V_1 \otimes V_2 \otimes \dots \otimes V_d$ is a representation of $G$. Suppose $v_i,w_i \in V_i$ are weight vectors of weights $\lambda_i, \mu_i$. Let $v = v_1 \otimes v_2 \otimes \dots \otimes v_d$ and $w = w_1 \otimes \dots \otimes w_d$. Clearly, the weight of $v$ is $(\lambda_1,\lambda_2,\dots,\lambda_d) := \underline{\lambda}$, and the weight of $w$ is $(\mu_1,\dots,\mu_d):= \underline{\mu}$. Specializing the discussion to tensor actions, we get:

\begin{corollary}  \label{uncramped.tensor}
Consider the tensor action of $\SL(V_1) \times \SL(V_2) \times \dots \times \SL(V_d)$ on $V_1 \otimes V_2 \otimes \dots \otimes V_d$. Suppose $\B_i$ is a basis for $V_i$ and we make all the standard choices for compact real form, tori etc with respect to  the basis $\B_i$. Let $v = b_1 \otimes b_2 \otimes \dots \otimes b_d$ and $w = b_1' \otimes b_2' \otimes \dots \otimes b_d'$ with $b_i,b_i' \in \B_i$ for all $i$. Suppose for at least two choices of $i \in \{1,2,\dots,d\}$, we have $b_i \neq b_i'$. Then $v$ and $w$ are weight vectors whose weights are not root adjacent. 
\end{corollary}

\section{Moment map and a criterion for closed orbits} 
In order to be able to use Theorem~\ref{main} effectively, we would need to prove that an orbit is closed. A criterion for detecting whether an orbit is closed is interesting by itself, and a good criterion could have a range of applications in both pure and applied mathematics. We approach the problem via the moment map, which suffices for our purposes. It is an interesting problem to understand whether the criterion we propose (see Theorem~\ref{crit.co}) has a suitable analogue in positive characteristic. We first define the moment map.

\begin{definition}
Let $V$ be a representation of a connected complex reductive group $G$. Let $K$ be a maximal compact group of $G$, and let $\g$ denote the Lie algebra of $G$. Let $\left< -, - \right>$ be a $K$-invariant positive definite hermitian form. The moment map $\mu_G : V \rightarrow \mathfrak{g}^*$ is defined by $\mu_G(v)(X) = \left<Xv, v \right>$ for $v \in V$ and $X \in \g$.
\end{definition}

\begin{proposition} [Kempf-Ness]
Suppose $\mu_G(v) = 0$, then the orbit $G \cdot v$ is closed.
\end{proposition}

An even stronger statement holds, namely that every closed orbit contains a unique $K$-orbit at which the moment map vanishes. This is precisely why the GIT quotient $X\sslash G$ agrees with the symplectic reduction $\mu^{-1}(0) / K$, which is known as the Kempf-Ness theorem. We refrain from getting into this beautiful subject, and refer to \cite{KN, Mumford} for details.

Now, we turn towards proving a criterion for the vanishing of the moment map in the language of root systems.

\begin{proposition}
Suppose $V$ is a representation of a connected complex reductive group $G$. Let $K$ be a compact real form, and let $T$ be a maximal torus, and let $\left<-,-\right>$ be a $(K,T)$-compatible form. Assume that $v \in V$. Let $v = \sum_{\lambda \in \Supp(v)} v_{\lambda}$ be its weight decomposition. Suppose
\begin{enumerate}
\item $\Supp(v)$ is \uncramped\  (see Definition~\ref{uncramped}).
\item $\sum_{\lambda \in \Supp(v)} ||v_{\lambda}||^2 \lambda = 0$. 
\end{enumerate}
Then, $\mu_G(v) = 0$, and hence the orbit of $v$ is closed.
\end{proposition}

\begin{proof}
Look at the root space decomposition $\g = \tn \oplus\bigoplus_{\alpha \in \Phi} \g_\alpha$. We want to show that $\mu_G(v)(X) = 0$ for all $X \in \g$. Since $\mu_G(v)$ is a linear map from $\g$ to $\C$, it will suffice to show $\mu_G(v)(X) = 0$ separately for $X \in \tn$ and $X \in \g_\alpha$ for each $\alpha \in \Phi$.

Suppose $X \in \g_\alpha$. Then for each $\lambda \in \Supp(v)$, we have $X \cdot v_{\lambda} \in V_{\lambda + \alpha}$. We know that $\lambda + \alpha \notin \Supp(v)$ because $\Supp(v)$ is \uncramped. Since the form is $(K,T)$-compatible, we know that weight spaces are orthogonal. So, $X \cdot v_{\lambda}$ is orthogonal to $v$. Hence $X \cdot v = \sum_{\lambda \in \Supp(v)} X \cdot v_{\lambda}$ is orthogonal to $v$, i.e., $\left< X \cdot v, v \right> = 0$. In other words, $\mu_G(v)(X) = 0$.

Now, suppose $X \in \tn$. Then for each $\lambda \in \Supp(v)$, we have $X \cdot v_{\lambda} = \lambda(X) v_\lambda$. Now, observe that 
$$
\left< X\cdot v_{\lambda}, v \right> = \left< \lambda(X) v_\lambda, v \right> = \left<\lambda(X) v_\lambda, v_{\lambda} \right> = \lambda(X) ||v_{\lambda}||^2.
$$

Thus, we have 
$$
\left< X \cdot v, v \right> = \sum_{\lambda \in \Supp(v)} \left< X\cdot v_{\lambda}, v \right> = \sum_{\lambda \in \Supp(v)} \lambda(X) ||v_{\lambda}||^2 = 0.
$$

The last equality of course follows from Condition $(2)$. Hence, we have $\mu_G(v)(X) = 0$ for $X \in \tn$. Thus, we conclude $\mu_G(v) = 0$, and consequently that the orbit $G \cdot v$ is closed.
\end{proof}

\begin{theorem} \label{crit.co}
Suppose $W$ is a representation of a connected complex reductive group $G$ and $w \in W$. Let $W = \bigoplus\limits_{j \in J} W_j$ be a decomposition into subrepresentations. Take a $(K,T)$-compatible form on each $W_j$, and let $\left<-,-\right>$ denote their direct sum form on $W$. Let $w = \sum_{j \in J} {w_j}$ with $w_j \in W_j$. Further, write $w_j = \sum_{\lambda \in \Supp(w_j)} w_{j,\lambda}$ be the weight decomposition for each $w_j$. Suppose 
\begin{enumerate}
\item $\Supp(w_j)$ is \uncramped\ for all $j$;
\item $\sum_{j} \sum_{\lambda \in \Supp(w_j)} ||w_{j,\lambda}||^2 \lambda = 0$. 
\end{enumerate}
Then, the orbit of $w$ is closed. 
\end{theorem}

\begin{proof}
We want to show that $\mu_G(w)(X) = 0$ for all $X \in \g$. Again, it suffices to show it separately for $X \in \tn$ and $X \in \g_\alpha$ for each $\alpha \in \Phi$. The argument for $\tn$ is the same as the previous proposition. 

Now suppose $X \in \g_\alpha$. For all $j$, we have $\left<Xw_j,w_j\right> = 0$ by repeating the argument from the previous proposition, as $\Supp(w_j)$ is uncramped. Since the irreducibles $W_j$ are orthogonal by construction of the form, this shows that $\left<Xw,w\right> = 0$ as required.
\end{proof}

\section{Cubic forms}
Let us set up the situation for this section. Let $V$ be a vector space of dimension $3n$, and let a basis for $V$ be $\B = \{x_i,y_i,z_i\}_{1 \leq i \leq n}$ .  Consider $W = \Sym^3(V)^{\oplus 4}$, and let 
$$
w = \Big(\sum_i x_i^2 z_i, \sum_i y_i^2 z_i, \sum_i \alpha_i x_iy_iz_i\Big),
$$

where $\alpha_i$ are distinct complex numbers with $|\alpha_i| = 1$ and for all $i \neq j$, $\alpha_i \neq \pm \alpha_j$. There is a natural action of $\SL(V)$ on $\Sym^3(V)$, and hence on $W$. We will write $w = (w_1,w_2,w_3)$ where $w_1 =  \sum_i x_i^2 z_i$, $w_2 = \sum_i y_i^2 z_i$ and $w_3 = \sum_i \alpha_i x_iy_iz_i.$

\begin{proposition} \label{closed w}
The orbit $\SL(V) \cdot w$ is closed.
\end{proposition}

Let us define a map $ \phi: (\C^*)^n \rightarrow \SL(V)$. To define the map, it suffices to understand how $\phi(t = (t_1,\dots,t_n))$ acts on the basis $\{x_i,y_i,z_i\}_{1 \leq i \leq n}$. Define $\phi$ by $\phi(t) \cdot x_i = t_i x_i$, $\phi(t) \cdot y_i = t_i y_i$  and $\phi(t) \cdot z_i = t_i^{-2} z_i$. Let $H := \phi((\C^*)^n)$.

\begin{proposition} \label{stab.compute}
We have ${\rm Stab}_{\SL(V)}(w) = H$.
\end{proposition}

It is also easy to see that $H$ is a closed subgroup of $G$. It is also reductive because it is a torus. It is indeed necessary that the stabilizer is closed and reductive to be able to apply Theorem~\ref{main}, as we will do in the proof of Theorem~\ref{lbsln}.

We postpone the proofs Proposition~\ref{closed w} and Proposition~\ref{stab.compute} and complete the proof of Theorem~\ref{lbsln}.

Consider the $n+1$-dimensional subspace $U \subset \Sym^3(V)$ spanned by $\{x_1z_2^2, x_2z_3^2,\dots,x_nz_1^2,x_1^3\}$. This is an invariant subspace under the action of $H \subset \SL(V)$ described in the previous section.

\begin{lemma}
We have $\beta_H(U) \geq \sigma_H(U) \geq  \textstyle\frac{2}{3}(4^n - 1)$.
\end{lemma}

\begin{proof}
The basis $\mathcal{E} = (x_1z_2^2, x_2z_3^2,\dots,x_nz_1^2,x_1^3)$ is a weight basis, and $M_{\mathcal{E}(W)} = M$, the matrix in Section~\ref{prelim}. The lemma now follows from Proposition~\ref{tor.inv.M}.
\end{proof}

\begin{corollary}
We have $\beta_H(\Sym^3(V)) \geq \sigma_H(\Sym^3(V)) \geq \textstyle\frac{2}{3}(4^n - 1)$.
\end{corollary}

\begin{proof}
This follows from Proposition~\ref{tor.inv.compare} since $U$ is a subrepresentation of $\Sym^3(V)$ for the action of $H$.
\end{proof}

\begin{proof} [Proof of Theorem~\ref{lbsln}]
Let $G = \SL(V)$. Recall $w \in \Sym^3(V)^{\oplus 3}$ from the previous section such that ${\rm Stab}_G(w) = H$. Thus, by Theorem~\ref{main} and the above corollary, we have 
$$
\beta_G(\Sym^3(V)^{\oplus 3} \oplus \Sym^3(V)) 
\geq \sigma_G(\Sym^3(V)^{\oplus 3} \oplus \Sym^3(V)) \geq \sigma_H(\Sym^3(V)) \geq  \textstyle\frac{2}{3}(4^n - 1).
$$  
\end{proof}

\begin{remark}
If instead of $w$, one takes $(\sum_i x_i^2 z_i, \sum_i y_i^2 z_i) \in \Sym^3(V)^{\oplus 2}$, then this also has a closed orbit. However, its stabilizer is not the torus $H$ (defined above), but rather a finite extension of it.  With some additional work, this can be used to show exponential lower bounds for $\Sym^3(V)^{\oplus 3}$ (instead of $\Sym^3(V)^{\oplus 4}$ as stated in Theorem~\ref{lbsln}). However, we feel that this modest improvement does not warrant the additional discussion on how to deal with finite extensions of tori, so we omit it.
\end{remark}

\subsection{Closedness of orbit}
The strategy is to apply Theorem~\ref{crit.co}. But before proceeding to check the hypothesis, we need a little groundwork.

\begin{definition} [type of a monomial]
Every monomial in the basis $\mathcal{B}$ can be written as $b_1^{a_1}b_2^{a_2} \dots b_k^{a_k}$, where the $b_i$ represent distinct elements in the basis $\mathcal{B}$, and $a_1 \geq a_2 \geq \dots  \geq a_k > 0$. We define its type to be $(a_1,\dots,a_k)$.
\end{definition}

\begin{example}
The types of $x_i^2z_i$ and $y_j^2z_j$ are $(2,1)$, whereas the type of $x_iy_iz_i$ is $(1,1,1)$.
\end{example}

\begin{lemma}
There exists a $\B$-compatible form on $\Sym^3(V)$. For any $\B$-compatible form $\Sym^3(V)$, all the monomials of a fixed type have the same norm.
\end{lemma}

\begin{proof}
Recall that a $\B$-compatible form is a $(K_{\GL(V),\B},T_{\GL(V),\B})$-compatible form. Since we have an action of $\GL(V)$ on $\Sym^3(V)$, there is a $\B$-compatible form on $\Sym^3(V)$ by Proposition~\ref{com.inv.exist}. A permutation of the basis vectors in $\B$ is an element of $K_{\GL(V),\B}$, and hence such permutations preserve norm. To conclude, notice that the monomials of a fixed type are related by such permutations.
\end{proof}

We can now prove Proposition~\ref{closed w}.

\begin{proof} [Proof of Proposition~\ref{closed w}]
We want to show that $w$ satisfies the hypothesis of Theorem~\ref{crit.co}. Recall that $w = (w_1,w_2,w_3)$ where $w_1 =  \sum_i x_i^2 z_i$, $w_2 = \sum_i y_i^2 z_i$ and $w_3 = \sum_i \alpha_i x_iy_iz_i.$ Note that these are the weight space decompositions $w_j = \sum_{\lambda \in \Supp (w_j)} w_{j,\lambda}$. 

We want to check that the hypothesis of Theorem~\ref{crit.co} is satisfied. 
To check condition $(1)$ of Theorem~\ref{crit.co}, we need to check that each $\Supp (w_j)$ is \uncramped. But observe from the weight decompositions that $\Supp (w_1) = \{2\wt{x}_i + \wt{z}_i\}_{1 \leq i \leq n}$, $\Supp (w_2) = \{2\wt{y}_i + \wt{z}_i\}_{1 \leq i \leq n}$ and $\Supp (w_3) = \{\wt{x}_i + \wt{y}_i + \wt{z}_i\}_{1 \leq i \leq n}$. It is clear that these are \uncramped\  from the description of the root system $\Phi$ in Remark~\ref{b-rootsys}.

To check condition~$(2)$, we need to do an explicit computation that requires us to be able to compute the norms of $x_i^2z_i$, $y_i^2z_i$ and $x_iy_iz_i$. From the above lemma, we know that there is a $\B$-compatible form on each copy of $\Sym^3(V)$. We insist that we use the same form on each copy and then use the direct sum form on $\Sym^3(V)^{\oplus 3}$. 

The above lemma tells us that all monomials of a certain type have the same norm. Let us suppose that the monomials of type $(2,1)$ (e.g., $x_i^2z_i$ and $y_j^2z_j$) have norm $M$ and monomials of type $(1,1,1)$ (e.g., $x_iy_iz_i)$ have norm $N$. 

We compute $\sum_{\lambda \in \Supp (w_1)} ||w_{1,\lambda}||^2 \lambda$.
\begin{align*}
\sum_{\lambda \in \Supp (w_1)} ||w_{1,\lambda}||^2 \lambda & = \sum_{i=1}^n ||x_i^2z_i||^2 (2\wt{x}_i + \wt{z}_i) \\
& = \sum_i M^2 (2\wt{x}_i + \wt{z}_i)
\end{align*}

Similarly, we have 
$$
\sum_{\lambda \in \Supp (w_2)} ||w_{2,\lambda}||^2 \lambda  = \sum_{i} M^2 (2\wt{y}_i + \wt{z}_i), 
$$

and 

\begin{align*}
\sum_{\lambda \in \Supp (w_3)} ||w_{3,\lambda}||^2 \lambda &= \sum_i ||x_iy_iz_i||^2 (\wt{x}_i + \wt{y}_i + \wt{z}_i)  \\
& = N^2 (\sum_i \wt{x}_i + \wt{y}_i + \wt{z}_i).
\end{align*}

Hence, we have 
$$
\sum_{j =1}^3\sum_{\lambda \in \Supp (w_j)} ||w_{j,\lambda}||^2 \lambda = (2M^2 + N^2) (\sum_i \wt{x}_i + \wt{y}_i + \wt{z}_i) =(2M^2 + N^2) \sum_{b \in \B} \wt{b} = 0.
$$
The last equality follows from Remark~\ref{B-defns}, as we are working with $\SL(V)$. Hence, $w$ satisfies the hypothesis of Theorem~\ref{crit.co}, so the orbit of $w$ is closed.
\end{proof}

\subsection{Computation of stabilizer}
Now, we turn towards computing the stabilizer. We will proceed in steps. 

\begin{lemma}
Suppose $g \in SL(V)$ such that $g \cdot w_1 = w_1$. Then $g \cdot x_i = c_i x_{\sigma(i)}$ for some permutation $\sigma$ of $\{1,2,\dots,n\}$ and non-zero scalars $c_i$.
\end{lemma}

\begin{proof}
The space of partial derivatives of $w_1$ is $\left< x_1^2,\dots,x_n^2, x_1z_1, \dots ,x_nz_n\right>$. This must be preserved by $g$. The squares in the space of partial derivatives are of the form $d_ix_i^2$ for some nonzero scalars $d_i$. Thus the image of $x_i$ under the action of $g$ must be a scalar multiple of $x_j$ for some $j$. Since $g$ is invertible, the lemma follows.
\end{proof}

\begin{corollary} \label{x2z}
Suppose $g \in {\rm Stab}_{\SL(V)} (w_1)$. Then for some permutation $\sigma$, we must have $g \cdot x_i = c_ix_{\sigma(i)}$ and $g \cdot z_i = c_i^{-2}z_{\sigma(i)}$ for some scalars $c_i$.
\end{corollary}

\begin{proof}
From the above lemma, we already know that $g \cdot x_i = c_i x_{\sigma(i)}$ for some permutation $\sigma$ and scalars $c_i$. Hence, we have 
$$
\sum_i (c_ix_{\sigma(i)})^2 (g \cdot z_i)  = g \cdot w_1 = w_1 =  \sum_i x_i^2 z_i = \sum_i x_{\sigma(i)}^2 z_{\sigma(i)}.
$$
Thus, we have 
$$
\sum_i x_{\sigma(i)}^2 (c_i^2 g \cdot z_i - z_{\sigma(i)}) = 0.
$$

Observe that monomials of degree $3$ in $\{x_i,y_i,z_i\}_{1 \leq i \leq n}$ are a basis for $\Sym^3(V)$. Now, for any $p,q \in V$, $x_i^2 p$ and $x_j^2 q$ do not have any monomials in common. Hence, we must have $x_{\sigma(i)}^2 (c_i^2 g \cdot z_i - z_{\sigma(i)}) = 0$ for all $i$. Hence, for all $i$, we must have $c_i^2 g \cdot z_i - z_{\sigma(i)} = 0$ or equivalently $g \cdot z_i = c_i^{-2} z_{\sigma(i)}$ as required. 
\end{proof}

We can do a similar analysis for $w_2$, and we get:

\begin{lemma} \label{y2z}
Suppose $g \in {\rm Stab}_{\SL(V)} (w_2)$. Then for some permutation $\pi$ and scalars $d_i$, we have $g \cdot y_i = d_iy_{\pi(i)}$ and $g \cdot z_i = d_i^{-2}z_{\pi(i)}$.
\end{lemma}

\begin{corollary}
Suppose $g \in {\rm Stab}_{\SL(V)} (w_1,w_2)$. Then for some permutation $\sigma$ and scalars $c_i$, we have $g(x_i) =  c_i x_{\sigma(i)}$, $g(y_i) = \pm  c_i y_{\sigma(i)}$ and $g(z_i) = c_i^{-2} z_{\sigma(i)}$.
\end{corollary}

\begin{proof}
Suppose $g \in {\rm Stab}_{\SL(V)} (w_1,w_2)$. Then from Corollary~\ref{x2z}, we know that there is a permutation $\sigma$ and scalars $c_i$ such that $g(x_i) = c_i x_{\sigma(i)}$ and $g(z_i) = c_i^{-2}z_{\sigma(i)}$. By Lemma~\ref{y2z}, there is a permutation $\pi$ and scalars $d_i$ such that $g(y_i) = d_i y_{\pi(i)}$ and $g(z_i) = d_i^{-2}z_{\pi(i)}$. 

Thus, we have $g\cdot z_i = c_i^{-2} z_{\sigma(i)} = d_i^{-2} z_{\pi(i)}$ for all $i$. Hence, we must have $\sigma = \pi$ and $d_i = \pm c_i$.
\end{proof}

\begin{proof} [Proof of Proposition~\ref{stab.compute}]
Suppose $g \in Stab(w_1,w_2,w_3)$. Then since $g \in Stab(w_1,w_2)$, we know that there is a permutation $\sigma$ and scalars $c_i$ such that $g(x_i) =  c_i x_{\sigma(i)}$, $g(y_i) = \pm  c_i y_{\sigma(i)}$ and $g(z_i) = c_i^{-2} z_{\sigma(i)}$

In particular, this means that $g \cdot x_iy_iz_i = \pm x_{\sigma(i)}y_{\sigma(i)}z_{\sigma(i)}$. But now $g$ also fixes $w_3 = \sum_i \alpha_i x_iy_iz_i$. However, we have
$$
\sum_i \pm \alpha_i  x_{\sigma(i)}y_{\sigma(i)}z_{\sigma(i)} = g \cdot w_3 = w_3 = \sum_i \alpha_i x_iy_iz_i
$$

This means that $\pm\alpha_i =  \alpha_{\sigma(i)}$. But recall that the choice of $\alpha_i$'s was such that $\alpha_i \neq \pm \alpha_j$ for all $i \neq j$. This means that $\sigma$ is the identity permutation, and further that we must have $g \cdot x_iy_iz_i = x_iy_iz_i$. Hence, this implies $g \cdot y_i = c_iy_i$.

Thus we must have $g \cdot x_i = c_ix_i$, $g \cdot y_i = c_iy_i$ and $g\cdot z_i = c_i^{-2}z_i$. In other words, $g \in H$. Conversely, it is easy to observe that $H \subseteq Stab(w)$.
\end{proof}

\section{Tensor actions}
Let $U,V,W$ be $3n$-dimensional vector spaces with basis $\B_u = \{u_1^k,u_2^k,u_3^k\}_{1 \leq k \leq n}$, $\B_v = \{v_1^k,v_2^k,v_3^k\}_{1 \leq k \leq n}$ and $\B_w = \{w_1^k,w_2^k,w_3^k\}_{1 \leq k \leq n}$ respectively. 

Let 
\begin{align*}
F_1 &= \sum_{k=1}^n u_1^k v_2^k w_3^k + u_2^k v_3^k w_1^k + u_3^k  v_1^k w_2^k \\
G_1 &= \sum_{k=1}^n \alpha_k u_1^k v_2^k w_3^k + \beta_k u_2^k v_3^k w_1^k + \gamma_k u_3^k  v_1^k w_2^k \\
F_2 &= \sum_{k=1}^n u_2^k v_1^k w_3^k + u_1^k v_3^k w_2^k + u_3^k  v_2^k w_1^k \\
G_2 &= \sum_{k=1}^n \alpha_k u_2^k v_1^k w_3^k + \beta_k u_1^k v_3^k w_2^k + \gamma_k u_3^k  v_2^k w_1^k \\
F_3 &= \sum_{k=1}^n u_1^k v_1^k w_3^k + u_2^k v_3^k w_2^k + u_3^k  v_1^k w_1^k \\
G_3 &= \sum_{k=1}^n \alpha_k u_1^k v_1^k w_3^k + \beta_k u_2^k v_3^k w_2^k + \gamma_k u_3^k  v_1^k w_1^k  \\
F_4 &= \sum_{k=1}^n u_2^k v_2^k w_3^k + u_1^k v_3^k w_1^k + u_3^k  v_2^k w_2^k \\
G_4 &= \sum_{k=1}^n \alpha_k u_2^k v_2^k w_3^k + \beta_k u_1^k v_3^k w_1^k + \gamma_k u_3^k  v_2^k w_2^k, 
\end{align*}

where $\alpha_k,\beta_k,\gamma_k$ are a collection of distinct scalars in $\C$ with unit norm. Consider
$$
\underline{F} = (F_1,G_1,F_2,G_2,F_3,G_3,F_4,G_4) \in (U \otimes V \otimes W)^8.
$$

The approach will be the same as cubic forms. First, we show:
\begin{proposition} \label{orbit.closed.tensor}
The orbit of $\underline{F}$ for the action of $\SL(U) \times \SL(V) \times \SL(W)$ is closed.
\end{proposition}

Next, we compute the stabilizer. Let us define a map $\phi_U:  ((\C^*)^3)^n \rightarrow \GL(U)$. To define such a map it suffices to understand the action of $t = (p_1,q_1,r_1,p_2,q_2,r_2,\dots,p_n,q_n,r_n)$ on each basis vector $b \in \B_u$. The map $\phi_U$ is defined by 
$$
\phi_U(t) u_1^k = p_k u_1^k, \phi_U(t) u_2^k = p_k u_2^k \text{ and } \phi_U(t) u_3^k = (q_kr_k)^{-1} u_3^k.
$$ 
Similarly define $\phi_V: ((\C^*)^3)^n \rightarrow \GL(V)$ by 
$$
\phi_V(t) v_1^k = q_k v_1^k, \phi_V(t) v_2^k = q_k v_2^k \text{ and } \phi_V(t) v_3^k = (p_kr_k)^{-1} v_3^k.
$$
 Finally, define $\phi_W: ((\C^*)^3)^n \rightarrow \GL(W)$ by 
 $$
 \phi_W(t) w_1^k = r_k w_1^k, \phi_W(t) w_2^k = r_k w_2^k \text{ and } \phi_W(t) w_3^k = (p_kq_k)^{-1} w_3^k.
 $$
 
 Let $\phi = (\phi_U,\phi_V,\phi_W): ((\C^*)^3)^n \rightarrow \GL(U) \times \GL(V) \times \GL(W)$. Let $H$ denote the image of $\phi$. Then, we have:
  
\begin{proposition} \label{tensor-stab-compute}
We have ${\rm Stab}_{\GL(U) \times \GL(V) \times \GL(W)} (\underline{F}) = H$.
\end{proposition}

Again, it is easy to check that $H$ is a closed subgroup of $\GL(U) \times \GL(V) \times \GL(W)$. It is also reductive because it is a torus. The reader perhaps has noticed that we have computed the stabilizer in $\GL(U) \times \GL(V) \times \GL(W)$ rather than the stabilizer in $\SL(U) \times \SL(V) \times \SL(W)$. There are several ways to fix this, and we indicate one of them.

Consider the group 
$$
J := \{ (g_1,g_2,g_3) \in \GL(U) \times \GL(V) \times \GL(W) \ |\ \det(g_1) \det(g_2) \det(g_3) = 1\}.
$$
Indeed, the first thing to observe is that $H \subset J$. Now, we claim that the orbits of $J$ and the orbits of $\SL(U) \times \SL(V) \times \SL(W)$ in $U \otimes V \otimes W$ are the same. Let $h = (g_1,g_2,g_3) \in J$. Since $\det(g_1)\det(g_2)\det(g_3) = 1$, we can choose $c_1,c_2,c_3 \in \C$ with $c_1c_2c_3 = 1$ such that $\det(c_ig_i) = 1$. Thus, we have $h \cdot v = (c_1g_1,c_2g_2,c_3g_3) \cdot v$ for any $v \in U \otimes V \otimes W$. But $(c_1g_1,c_2g_2,c_3g_3) \in \SL(U) \times \SL(V) \times \SL(W)$, so this means that the $J$-orbit of $v$ is contained in the $\SL(U) \times \SL(V) \times \SL(W)$-orbit of $v$. On the other hand, $J \supseteq \SL(U) \times \SL(V) \times \SL(W)$, so the orbits must be the same. The same argument works for $(U \otimes V \otimes W)^{\oplus m}$. Further observe that the quotient $GL(U) \times GL(V) \times \GL(W)/J = \C^*$, which is affine. Since $J$ is clearly a closed subgroup of $\GL(U) \times \GL(V) \times \GL(W)$, by Matsushima's criterion (see Theorem~\ref{Matsu}) we conclude that $J$ is reductive. We summarize the above discussion as follows:

\begin{proposition}
The $J$-orbit of $\underline{F}$ is closed. Further, the stabilizer of $\underline{F}$ in $J$ is $H$. Moreover $J$ is a reductive group.
\end{proposition}

Further, since orbits of $J$ are the same as the orbits of $\SL(U) \times \SL(V) \times \SL(W)$, we also have that the invariant rings are equal, i.e,

\begin{corollary} \label{J.inv.equal}
We have $\C[U \otimes V \otimes W]^{\SL(U) \times \SL(V) \times \SL(W)} = \C[U \otimes V \otimes W]^J.$
\end{corollary}

Consider the action of $H$ on $U \otimes V \otimes W$. Let $L$ denote the subspace spanned by $\mathcal{E} = \{u_1^1v_1^1w_1^1\} \cup \{u_1^{k+1} v_3^k w_3^k, u_3^kv_1^{k+1}w_3^k, u_1^kv_1^kw_3^{k+1}\}_{1 \leq k \leq n-1} \cup \{u_3^nv_3^nw_3^n \}$. Now, it is clear that for the action of $H$ on $L$, the set $\mathcal{E}$ is a weight basis, and further one can check that $M_{\mathcal{E}}(L) = N$, the matrix in Section~\ref{prelim}. Hence, from Proposition~\ref{tor.inv.N}, we obtain:

\begin{corollary}
We have $$
\beta_H(U \otimes V \otimes W) \geq \sigma_H(U \otimes V \otimes W) \geq \sigma_H(L) \geq 4^n-1.
$$
\end{corollary}

\begin{proof} [Proof of Theorem~\ref{tensor-lbs}]
Proceed in exactly the same fashion as the proof of Theorem~\ref{lbsln} to obtain the required lower bounds on $\sigma_J((U \otimes V \otimes W)^{\oplus 9})$ and $\beta_J((U\otimes V \otimes W)^{\oplus 9})$. Then using Corollary~\ref{J.inv.equal}, we conclude that the same lower bounds hold for $\SL(U) \times \SL(V) \times \SL(W)$.
\end{proof}

\subsection{Closedness of orbit}
This section is devoted to the proof of Proposition~\ref{orbit.closed.tensor}. The strategy will again be to use Theorem~\ref{crit.co}. We have the basis $\B_u,\B_v$ and $\B_w$ for $U,V,W$ respectively. For $\SL(U) \times \SL(V) \times \SL(W)$, we choose $K:= K_{\B_u} \times K_{\B_v} \times K_{\B_w}$ for a compact real form and $T = T_{\B_u} \times T_{\B_v} \times T_{\B_w}$ for a maximal torus.

Observe that $\B = \{b_u \otimes b_v \otimes b_w \ | \ b_u \in \B_u, b_v \in \B_v, b_w \in \B_w\}$ is a basis for $U \otimes V \otimes W$. Consider the hermiitan form on $U \otimes V \otimes W$ given by asking for $\B$ to be an orthonormal basis. It is easy to check that this form is $(K,T)$-compatible.

\begin{proof} [Proof of Proposition~\ref{orbit.closed.tensor}]
We use the form described above for each copy of $U \otimes V \otimes W$ and take the direct sum form. In order to use Theorem~\ref{crit.co}, the first step is to check that the supports $\Supp(F_d)$ and $\Supp(G_d)$ are uncramped. Let us only indicate the proof for $F_1$, as the other cases are similar. The defining decomposition of $F_1$ is its weight decomposition. It has three types of terms $u_1^k v_2^k w_3^k$, $u_2^k v_3^k w_1^k$, and $u_3^k  v_1^k w_2^k$. We want to show that the support is uncramped. So, for any two such terms, we need to show that their weights are not root adjacent. But this follows easily from Corollary~\ref{uncramped.tensor}. 

Let us now check the second condition in Theorem~\ref{crit.co} i.e., we want:
$$
\sum_{d =1}^4 \left(\sum_{\lambda \in \Supp(F_d)} ||(F_d)_\lambda||^2 \lambda + \sum_{\mu \in \Supp(G_d)} ||(G_d)_{\lambda}||^2 \mu\right) = 0.
$$

The defining decompositions of $F_d$ and $G_d$ are weight decompositions. All the coefficients appearing in $F_d$ and $G_d$ have absolute value $1$. Further, observe that $\Supp(F_d) = \Supp(G_d)$. Thus we have 
\begin{align*}
\sum_d \left(\sum_{\lambda \in \Supp(F_d)} ||(F_d)_\lambda||^2 \lambda + \sum_{\mu \in \Supp(G_d)} ||(G_d)_{\lambda}||^2 \mu\right) &= \sum_d \left(\sum_{\lambda \in \Supp(F_d)} \lambda + \sum_{\mu \in \Supp(G_d)}  \mu\right) \\
& = 2 \sum_d \left(\sum_{\lambda \in \Supp(F_d)} \lambda\right).
\end{align*}

Recall that $\wt{u}_i^k$ denotes the weight for $u_i^k$ for $\SL(U)$. Recall that $\sum_{i,k} \wt{u}_i^k = 0$ from Remark~\ref{B-defns}. Observe that each $u_i^k$ appears a total of $4$ times in all the terms of $T_1,T_2,T_3,T_4$. Similarly for $v_i^k$ and $w_i^k$. This means that
\begin{align*}
\sum_d \left(\sum_{\lambda \in \Supp(F_d)} \lambda\right) &= 4(\sum_{i,k} \wt{u}_i^k, \sum_{i,k} \wt{v}_i^k, \sum_{i,k} \wt{w}_i^k) \\
& = 0.
\end{align*}

Hence, the second condition of Theorem~\ref{crit.co} is satisfied for $\underline{F}$. This concludes the proof. 
\end{proof}

\subsection{Computation of Stabilizer}
In spirit, the computation is very similar to the computation for cubic forms in the previous section. However, we will need slightly different arguments for this.

Tensors of the form $a \otimes b \otimes c \in U \otimes V \otimes W$ are called rank $1$ tensors. 

\begin{lemma} \label{kruskal}
Suppose $T = \sum_{i =1}^r a_i \otimes b_i \otimes c_i \in U \otimes V \otimes W$, where $\{a_i\}, \{b_i\}, \{c_i\}$ are linearly independent collections of vectors in $U,V$ and $W$ respectively. Then this is the unique decomposition of $T$ into a sum of $r$ rank $1$ tensors. 
\end{lemma}

\begin{proof}
For $r = 1$, this is clear. For $r \geq 2$, this follows from Kruskal's theorem, see \cite{Kruskal}.
\end{proof}
The above lemma can also be proved by using just elementary linear algebra arguments without resorting to Kruskal's theorem.

\begin{lemma} 
Suppose $g \in \GL(U) \times \GL(V) \times \GL(W)$ fixes $T$ as in the previous lemma. Then $g$ must permute the terms $a_i \otimes b_i \otimes c_i$. 
\end{lemma}

\begin{proof}
Applying $g$ to the decomposition into a sum of $r$ rank $1$ tensors also yields a decomposition into a sum of $r$ rank $1$ tensors. Hence, by the above lemma, $g$ must permute the terms. 
\end{proof}

\begin{corollary}
Suppose $g \in \GL(U) \times \GL(V) \times \GL(W)$ fixes $F_1$, then $g$ must permute the terms in $F_1$.
\end{corollary}

\begin{corollary}
Suppose $g \in \GL(U) \times \GL(V) \times \GL(W)$ fixes $F_1$ and $G_1$, then $g$ must fix all the terms in $F_1$.
\end{corollary}

\begin{proof}
Any non-trivial permutation of the terms in $F_1$ does not fix $G_1$. Hence $g$ must fix all the terms. 
\end{proof}

Similar arguments hold for $F_2,F_3$ and $F_4$ as well. In summary, we obtain:

\begin{corollary} \label{stupidere}
Suppose $g \in \GL(U) \times \GL(V) \times \GL(W)$ fixes $\underline{F}$, then $g$ must fix all the terms in $F_1,F_2,F_3$ and $F_4$.
\end{corollary}

Let $I_k = \{u_i^k v_j^k w_3^k, u_i^k v_3^k w_j^k , u_3^k v_i^k w_j^k \}_{1 \leq i,j \leq 2}$. Then $\cup_k I_k$ are precisely the terms occuring in $F_1,F_2,F_3$ and $F_4$.

\begin{lemma}
Suppose $g = (g_u,g_v,g_w) \in \GL(U) \times \GL(V) \times \GL(W)$ fixes $I_k$. Then for some $p_k,q_k,r_k \in \C^*$, we have
\begin{align*}
& g_u(u_i^k) = p_k u_i^k \text{ for } i = 1,2 \text{ and } g_u(u_3^k) = (q_kr_k)^{-1} u_3^k, \\
& g_v(v_i^k) = q_k v_i^k \text{ for } i = 1,2 \text{ and } g_v(v_3^k) = (p_kr_k)^{-1} v_3^k, \\
& g_w(w_i^k) = r_k w_i^k \text{ for } i = 1,2 \text{ and } g_w(w_3^k) = (p_kq_k)^{-1} w_3^k.
\end{align*}
\end{lemma}

\begin{proof}
It is clear that if $g$ fixes $b_u \otimes b_v \otimes b_w$, then each $g_x$ must scale $b_x$ for each $x \in \{u,v,w\}$. So, we must have $g_u(u_1^k) = p_k u_1^k$, $g_v(v_1^k) = q_k v_1^k$ and $g_w(w_1^k) = r_k w_1^k$ for some $p_k,q_k,r_k \in \C^*$. Then, since $u_1^kv_1^kw_3^k \in I_k$ is fixed by $g$, we must have $g_w(w_3^k) = (p_kq_k)^{-1} w_3^k$. Since $u_1^kv_2^kw_3^k \in I_k$ is fixed by $g$, we must have $g_v (v_2^k) = q_k v_k$. Symmetric arguments complete the proof.
\end{proof}

\begin{proof} [Proof of Proposition~\ref{tensor-stab-compute}]
From Corollary~\ref{stupidere}, we conclude that if $g$ fixes $\underline{F}$, then it must fix all the terms in $\cup_k I_k$. From the previous lemma, one concludes that $g \in H$. Conversely, it is easy to check that $H$ fixes $\underline{F}$.
\end{proof}

\section{Concluding remarks}
It was pointed out to us by David Wehlau that for the adjoint representation (denoted ${\rm Ad}$) of a group $G$, a generic point has a closed orbit whose stabilizer is a maximal torus. This gives rise to a plethora of examples with exponential degree lower bounds. Take any representation $W$ of $G$ for which one can prove exponential degree lower bounds for invariants w.r.t a maximal torus. Then the same lower bound would also hold for the ring of $G$-invariants for ${\rm Ad} \oplus W$. 

The proof of Theorem~\ref{main} requires characteristic zero. The proof breaks down in positive characteristic. For example, in the proof of Lemma~\ref{need.char0}, we use Reynolds operators, which do not exist in positive characteristic. However, one can modify the arguments in a standard way to get similar statements for separating invariants. Also, we do not have a general technique to prove that an orbit is closed as there is no analog of moment map in positive characteristic. We can get around this by using the adjoint representation as discussed above. Hence, we can construct representations with exponential degree lower bounds even in positive characteristic (for example, the action of $\SL(V)$ on ${\rm Ad} \oplus \Sym^3(V)$).

Sometimes, one is interested in a specific representation that doesn't contain the adjoint representation as a direct summand. One such example is the case of tensor actions that we address in this paper. It remains a difficult open problem to prove exponential lower bounds in such cases in positive characteristic. The main issue happens to be the fact we do not know any criterion that can be used to show that an orbit is closed in positive characteristic.

\ \\[20pt]
\noindent{\sl Harm Derksen\\
Department of Mathematics\\
University of Michigan\\
530 Church Street\\
Ann Arbor, MI 48109-1043, USA\\
{\tt hderksen@umich.edu}}

\ \\[20pt]
\noindent{\sl Visu Makam\\
School of Mathematics\\
Institute for Advanced Study\\
Princeton, NJ 08540, USA\\
{\tt visu@ias.edu}}

\end{document}